\documentclass[a4paper,12pt]{article}

\usepackage{amsfonts}
\usepackage{amscd,color}
\usepackage{amsmath,amsfonts,amssymb,amscd}
\usepackage{indentfirst,graphicx,epsfig}
\usepackage{graphicx}
\input{epsf}
\usepackage{epstopdf}
\usepackage{caption}
\usepackage{ifpdf}
\usepackage{fullpage}
\usepackage{epsfig}
\usepackage{latexsym}
\usepackage{caption}
\usepackage{amsmath,amsthm}
\usepackage{amsfonts}
\usepackage{amssymb}
\usepackage{graphicx}
\usepackage{graphics}
\usepackage{bm}
\usepackage{color}
\usepackage{subfigure}
\usepackage[noend]{algpseudocode}
\usepackage{algorithmicx,algorithm}
\graphicspath{{\figs}}

\makeatletter

\newcommand{\Rmnum}[1]{\expandafter\@slowromancap\romannumeral #1@}
\makeatother
\makeatletter
\let\@fnsymbol\@arabic
\makeatother

\begin{document}
\newtheorem{theorem}{Theorem}[section]
\newtheorem{observation}[theorem]{Observation}
\newtheorem{corollary}[theorem]{Corollary}
\newtheorem{problem}[theorem]{Problem}
\newtheorem{question}[theorem]{Question}
\newtheorem{lemma}[theorem]{Lemma}
\newtheorem{proposition}[theorem]{Proposition}
\newtheorem{remark}[theorem]{Remark}

\newtheorem{definition}[theorem]{Definition}
\newtheorem{guess}[theorem]{Conjecture}
\newtheorem{claim}[theorem]{Claim}
\newtheorem{example}[theorem]{Example}
\makeatletter
  \newcommand\figcaption{\def\@captype{figure}\caption}
  \newcommand\tabcaption{\def\@captype{table}\caption}
\makeatother

\newtheorem{acknowledgement}[theorem]{Acknowledgement}

\newtheorem{axiom}[theorem]{Axiom}
\newtheorem{case}[theorem]{Case}
\newtheorem{conclusion}[theorem]{Conclusion}
\newtheorem{condition}[theorem]{Condition}
\newtheorem{conjecture}[theorem]{Conjecture}
\newtheorem{criterion}[theorem]{Criterion}
\newtheorem{exercise}[theorem]{Exercise}
\newtheorem{notation}[theorem]{Notation}
\newtheorem{solution}[theorem]{Solution}
\newtheorem{summary}[theorem]{Summary}
\newtheorem{fact}[theorem]{Fact}

\newcommand{\pp}{{\it p.}}
\newcommand{\de}{\em}
\newcommand{\mad}{\rm mad}

\newcommand*{\QEDA}{\hfill\ensuremath{\blacksquare}}  %? Ô¶??壬Ê??
\newcommand*{\QEDB}{\hfill\ensuremath{\square}}  %? Զ??�???

\newcommand{\qf}{Q({\cal F},s)}
\newcommand{\qff}{Q({\cal F}',s)}
\newcommand{\qfff}{Q({\cal F}'',s)}
\newcommand{\f}{{\cal F}}
\newcommand{\ff}{{\cal F}'}
\newcommand{\fff}{{\cal F}''}
\newcommand{\fs}{{\cal F},s}
\newcommand{\g}{\gamma}
\newcommand{\wrt}{with respect to }

\title {Monochromatic disconnection: Erd\H{o}s-Gallai-type problems and product graphs\footnote{Supported by NSFC No.11871034 and 11531011.}}

\renewcommand{\thefootnote}{\arabic{footnote}}

\author{\small Ping Li, ~ Xueliang Li\\
\small Center for Combinatorics and LPMC\\
\small  Nankai University, Tianjin 300071, China\\
\small qdli\underline{ }ping@163.com, ~ lxl@nankai.edu.cn\\
}

\date{}
\maketitle

\begin{abstract}
For an edge-colored graph $G$, we call an edge-cut $M$ of $G$ monochromatic if the edges of $M$ are colored with a same color.
The graph $G$ is called monochromatically disconnected if any two distinct vertices of $G$ are separated by a monochromatic edge-cut.
The monochromatic disconnection number, denoted by $md(G)$, of a connected graph $G$ is the maximum number of colors that are allowed
to make $G$ monochromatically disconnected.
In this paper, we solve the Erd\H{o}s-Gallai-type problems for the monochromatic disconnection, and give the monochromatic disconnection
numbers for four graph products, i.e., Cartesian, strong, lexicographic, and tensor products.\\[2mm]
\end{abstract}

\baselineskip16pt

\section{Introduction}
Let $G$ be a graph and let $V(G)$, $E(G)$ denote the vertex set and the edge set of $G$, respectively.
Let $|G|$ (also $v(G)$) denote the number of vertices of $G$. If there is no confusion, we use $n$ and $m$ to denote, respectively, the number of vertices and edges of
a graph, throughout this paper. For $v\in V(G)$, let $d_G(v)$ denote the degree of $v$ in $G$ and let $N_G(v)$ denote the neighbors of $v$ in $G$.
We call a vertex $v$ of $G$ a {\em t-degree} vertex of $G$ if $d_G(v)=t$.
Let $\delta(G)$ and $\Delta(G)$ denote the minimum and maximum degree of $G$, respectively.
For all other terminology and notation not defined here we follow Bondy and Murty \cite{B}.

For a positive integer $t$, we use $[t]$ to denote the set $\{1,2,\cdots,t\}$ of natural numbers.
For a graph $G$, let $\Gamma: E(G)\rightarrow [k]$ be an {\em edge-coloring} of $G$ that allows a
same color to be assigned to adjacent edges, and $\Gamma$ is also called a {\em k-edge-coloring} of $G$
since $k$ colors are used.
For an edge $e$ of $G$, we use $\Gamma(e)$ to denote the color of $e$. If $H$ is a subgraph of $G$,
we also use $\Gamma(H)$ to denote the set of colors used on all edges of $H$. Let $|\Gamma|$ denote the number of
colors in $\Gamma$. An edge-coloring $\Gamma$ of $G$ is {\em trivial} if $|\Gamma|=1$, otherwise, it is {\em non-trivial}.

The new concept of monochromatic disconnection of graphs, recently introduced in \cite{LL} by us, is actually motivated from the concepts of rainbow disconnection \cite{CDHHZ} and monochromatic connection \cite{CY,LW2} of graphs.
For an edge-colored graph $G$, we call an edge-cut $M$ a {\em monochromatic edge-cut} if the edges of $M$ are colored with a same color.
For two vertices $u,v$ of $G$, a {\em monochromatic uv-cut} is a monochromatic edge-cut that separates $u$ and $v$.
An edge-colored graph $G$ is {\em monochromatically disconnected} if any two vertices of $G$ has a monochromatic cut separating them.
An edge-coloring of $G$ is a {\em monochromatic disconnection coloring} ($MD$-coloring for short) if it
makes $G$ monochromatically disconnected. The {\em monochromatic disconnection number}, denoted by $md(G)$, of a connected graph $G$
is the maximum number of colors that are allowed to make $G$ monochromatically disconnected. An {\em extremal MD-coloring} of $G$ is
an $MD$-coloring that uses $md(G)$ colors. If $H$ is a subgraph of $G$ and $\Gamma$ is an edge-coloring of $G$,
we call $\Gamma$ an edge-coloring {\em restricted} on $H$.

For a $k$-edge-coloring of $G$ and an integer $j\in[k]$, a {\em j-induced edge set} is the set of edges of $G$ colored with color $j$.
We also call a $j$-induced edge set a {\em color-induced edge set}. Then an edge-coloring of a graph is an $MD$-coloring if any two vertices
can be separated by a color-induced edge set. We will use this method to verify whether an edge-coloring of a graph is an $MD$-coloring.

Let $K_n^-$ be a graph obtained from $K_n$ by deleting an arbitrary edge. $K_3$ is also called a {\em triangle}. We call a path $P$ a {\em t-path} if $|E(P)|=t$ and denote it by $P_t$.
Analogously, we call a cycle $C$ a {\em t-cycle} if $|C|=t$ and denote it by $C_t$.

Let $e=uv$ be an edge of $G$. If $d_G(u)$=1, then we call $u$ a {\em pendent vertex} and call $e$ a {\em pendent edge} of $G$.
A block $B$ of a graph $G$ is {\em trivial} if $B=K_2$, otherwise $B$ is {\em non-trivial}. The {\em union} of two graphs $G$ and $H$ is the graph $G\cup H$
with vertex set $V(G)\cup V(H)$ and edge set $E(G)\cup E(H)$.

The following results were proved in \cite{LL}, and they are useful in the sequel.

\begin{proposition}\cite{LL} \label{parallel}
Suppose $G$ is a connected graph that may have parallel edges but does not have loops.
Let $G'$ be the underling simple graph of $G$. Then $md(G)=md(G')$.
\end{proposition}

\begin{proposition}\cite{LL} \label{block}
If $G$ has $r$ blocks $B_1,\cdots, B_r$, then $md(G)=\sum_{i\in [r]}md(B_i)$. Furthermore,
\begin{enumerate}
\item $md(G)=n-1$ if and only if $G$ is a tree;
\item if $G$ is a unique cycle graph, then $n-2\geq md(G)\geq \left\lfloor\frac{n}{2}\right\rfloor$, with equality when $G$ is a cycle.
\end{enumerate}
\end{proposition}

\begin{proposition}\cite{LL} \label{key}
Let $D$ be a connected subgraph of a graph $G$. If $\Gamma$ is an $MD$-coloring of $G$, then $\Gamma$ is also an $MD$-coloring restricted on $D$.
\end{proposition}

\begin{lemma}\cite{LL} \label{sub}
If $H$ is a connected spanning subgraph of $G$, then $md(H)\geq md(G)$.
\end{lemma}

From this, one can deduce that $1\leq md(G)\leq n-1$ for a connected graph of order $n$, just by considering a spanning tree of $G$.

\begin{lemma}\cite{LL} \label{subb}
Let $H$ be the union of some graphs $H_1,\cdots,H_r$. If $\bigcap_{i\in[r]}E(H_i)\neq \emptyset$ and $md(H_i)=1$ for each $i\in[r]$, then $md(H)=1$.
\end{lemma}

\begin{lemma}\cite{LL}\label{kst}
If $G$ is $K_n$, $K_n^-$ or
$K_{n,t}$ where $n\geq2$ and $t\geq3$, then $md(G)=1$.
\end{lemma}

\begin{theorem} \cite{LL} \label{n-2}
If $G$ is a $2$-connected graph, then $md(G)\leq \left\lfloor\frac{n}{2}\right\rfloor$.
\end{theorem}

An edge-cut $M$ of $G$ is a {\em matching cut} if $M$ is a matching of $G$. A graph is called {\em matching immune} if it has no matching cut.

\begin{theorem}\cite{BFP} \label{mat-i}
If a graph $G$ is matching immune, then $e(G)\geq \left\lceil\frac{3}{2}(v(G)-1)\right\rceil$.
\end{theorem}

The four main graph products are Cartesian, strong, lexicographic, and tensor products.
Let $G$ and $H$ be two graphs and $V(G)\times V(H)=\{(u,v): u\in V(G)\mbox{ and }v\in V(H)\}$. The four graph products are defined as follows.

$\bullet$ The {\em Cartesian product} of $G$ and $H$, written as $G\Box H$, is the graph with vertex set $V(G)\times V(H)$,
in which two vertices $(u,v)$ and $(u',v')$ are adjacent if and only if $uu'$ is an edge of $G$ and $v=v'$, or $vv'$ is an edge of $H$ and $u=u'$.

$\bullet$ The {\em strong product} of $G$ and $H$, written as $G\boxtimes H$, is the graph with vertex set $V(G)\times V(H)$,
in which two vertices $(u,v)$ and $(u',v')$ are adjacent if and only if $uu'$ is an edge of $G$ and $v=v'$, or $vv'$ is an edge of $H$ and $u=u'$,
or $uu'$ is an edge of $G$ and $vv'$ is an edge of $H$.

$\bullet$ The {\em lexicographic product} of $G$ and $H$, written as $G\circ H$, is the graph with vertex set $V(G)\times V(H)$,
in which two vertices $(u,v)$ and $(u',v')$ are adjacent if and only if $uu'$ is an edge of $G$, or $u=u'$ and $vv'$ is an edge of $H$.

$\bullet$ The {\em tensor product} of $G$ and $H$, written as $G\ast H$, is the graph with vertex set $V(G)\times V(H)$,
in which two vertices $(u,v)$ and $(u',v')$ are adjacent if and only if $uu'$ is an edge of $G$ and $vv'$ is an edge of $H$.

\begin{proposition}\label{S-L}
For two connected graphs $G$ and $H$, we have
\begin{enumerate}
\item $G\boxtimes H$ is a connected spanning subgraph of $G\circ H$.
\item $G\boxtimes H=(G\Box H)\cup (G\ast H)$ and $E(G\Box H)\cap E(G\ast H)=\emptyset$.
\end{enumerate}
\end{proposition}

\begin{proposition}\cite{W} \label{Te}
If $G$ and $H$ are connected graphs, then $G\ast H$ is connected if and only at least one of $G$ and $H$ is not bipartite.
\end{proposition}

\section{Preliminaries}
Let $e$ and $e'$ be two edges of a graph $G$. We say that $e$ and $e'$ satisfy the {\em relation} $\theta$ if there exists a sequence
of subgraphs $G_1,\cdots, G_k$ of $G$ where each $G_i$ is either a triangle or a $K_{2,3}$, such that $e\in E(G_1)$ and $e'\in E(G_k)$  and $E(G_i)\cap E(G_{i+1})\neq\emptyset$ for $i\in[k-1]$. We denote $e\theta e'$ if $e$ and $e'$ satisfy the relation $\theta$. For a graph $G$, if any two edges $e$ and $e'$ of $G$ satisfy $e\theta e'$, then we call the graph $G$ is a {\em closure}.

\begin{lemma}\label{closure}
If $G$ is a closure, then $md(G)=1$.
\end{lemma}
\begin{proof}
Let $\Gamma$ be an extremal $MD$-coloring of $G$ and $e$ be an edge of $G$. For every edge $f$ of $G$, there is a sequence
of subgraphs $G_1,\cdots, G_k$ of $G$ such that $e\in E(G_1)$ and $f\in E(G_k)$, and there is an edge $f_i$ of $G$
such that $f_i\in E(G_i)\cap E(G_{i+1})$ for $i\in[k-1]$. Here each $G_i$ is either a $K_3$ or a $K_{2,3}$.
Since $md(K_3)=md(K_{2,3})=1$, all edges of $G_i$ are colored with a same color. Then $\Gamma(e)=\Gamma(f_1)=\cdots=\Gamma(f)$.
Therefore, each edge of $G$ is colored with color $\Gamma(e)$ under $\Gamma$, and hence $md(G)=1$.
\end{proof}

\begin{lemma}\label{secq}
Let $G$ be a connected graph and $v\in V(G)$. If $v$ is neither a pendent vertex nor a cut-vertex of $G$, then $md(G)\leq md(G-v)$.
\end{lemma}
\begin{proof}
Let $\Gamma$ be an extremal $MD$-coloring of $G$. Then $\Gamma$ is an $MD$-coloring restricted on $G-v$. If $\Gamma(G)- \Gamma(G-v)=\emptyset$,
then $md(G)=|\Gamma|=|\Gamma(G-v)|\leq md(G-v)$. Therefore, it is sufficient to show that $\Gamma(G)- \Gamma(G-v)=\emptyset$.
Otherwise let $e=vu$ be an edge of $E(G)- E(G-v)$ and $\Gamma(e)\notin \Gamma(G-v)$. Since $d_G(v)\geq2$, there is another edge incident with $v$, say $f=vw$.
Because $v$ is not a cut-vertex, there is a cycle $C$ of $G$ containing $e$ and $f$. Because $\Gamma$ is an $MD$-coloring restricted on $C$,
there are at least two edges in the monochromatic $uv$-cut of $C$ and one of them is $e$. Thus $f$ is in the monochromatic $uv$-cut, i.e.,
$\Gamma(e)=\Gamma(f)$. Then, there is no monochromatic $uw$-cut in $C$, a contradiction.
\end{proof}

Suppose $G$ is a connected graph and $S=\{v_1,\cdots,v_t\}$ is a set of vertices of $G$. Let $G_0=G$ and $G_i=G-\{v_1,\cdots,v_i\}$ for $i\in [t]$.
We call the vertex sequence $\gamma=(v_1,v_2,\cdots,v_t)$ a {\em soft-layer} if $d_{G_{i-1}}(v_i)\geq2$ and $G_i$ is connected for $i\in[t]$.
The following result can be derived from Lemma \ref{secq} directly.

\begin{lemma} \label{sequ}
Suppose $G$ is a connected graph and $S=\{v_1,\cdots,v_t\}$ is a set of vertices of $G$. If the vertex sequence $\gamma=(v_1,v_2,\cdots,v_t)$ is a soft-layer, then $md(G)\leq md(G_t)$.
\end{lemma}

\begin{lemma}\label{matching-cut}
If $G$ has a matching cut, then $md(G)\geq 2$.
\end{lemma}
\begin{proof}
Let $M$ be a matching cut of $G$. Let $\Gamma$ be an edge-coloring of $G$ obtained by coloring $M$ with color $1$ and coloring $E(G)- M$ with color $2$.
Then for any two vertices $u$ and $v$ of $G$, if $uv$ is not an edge of $G$ or $uv\notin M$, then $u,v$ are in different components of $G-(E(G)- M)$;
if $uv\in M$, then $u,v$ are in different components of $G-M$. Therefore, $\Gamma$ is an $MD$-coloring of $G$, and hence $md(G)\geq2$.
\end{proof}

\begin{lemma}\label{leqk}
For a connected graph $G$ and an integer $r$ with $1\leq r\leq md(G)$, there is an $MD$-coloring $\Gamma$ of $G$ such that $|\Gamma|=r$.
\end{lemma}
\begin{proof}
Suppose $\Gamma'$ is an extremal $MD$-coloring of $G$. Then $|\Gamma'|=md(G)$. Let $E_i$ be the $i$-induced edge set for $i\in[md(G)]$. Let $\Gamma$ be an edge-coloring obtained from $\Gamma'$ by recoloring $E'=\bigcup_{i=r}^{md(G)}E_i$ by $r$. Then $|\Gamma|=r$. We now show that $\Gamma$ is an $MD$-coloring of $G$. For two vertices $u,v$ of $G$, since $\Gamma'$ is an extremal $MD$-coloring of $G$, there is an $E_i$ such that $u,v$ are in different components of $G-E_i$. Let $E''=E_i$ if $i<r$ and $E''=E'$ if $i\geq r$. Then $u,v$ are in different components of $G-E''$. This implies $\Gamma$ is an $MD$-coloring of $G$.
\end{proof}

\begin{theorem}
For a connected graph $G$, $md(G)=1$ if $\delta(G)\geq \left\lfloor\frac{n}{2}\right\rfloor+1$, and the lower bound is sharp.
\end{theorem}
\begin{proof}
To prove $md(G)=1$, it is sufficient to prove $G$ is a closure.

In fact, any two adjacent edges of $G$ are either in a triangle or in a $K_{2,3}$, because
for any two adjacent edges $e_1=ab$ and $e_2=ac$, $d_G(b)+d_G(c)\geq 2\left\lfloor\frac{n}{2}\right\rfloor+2\geq n+1$,
and so either $bc$ is an edge of $G$ or $b$ and $c$ have at least three common vertices.

For two edges $e_1$ and $e_2$ of $G$, there is a path $P$ of $G$ with pendent edges $e_1$ and $e_2$. Since any two adjacent edges
of $P$ are in a $K_3$ or a $K_{2,3}$, $G$ is a closure. Therefore $md(G)=1$.

Now we show that the bound is sharp, i.e., we need to construct a graph $H$ with $\delta(H)=\left\lfloor\frac{n}{2}\right\rfloor$ and $md(H)\geq2$.
Let $A,B$ be two vertex-disjoint complete graphs with $V(A)=\{v_1,\cdots,v_{\left\lceil\frac{n}{2}\right\rceil}\}$ and
$V(B)=\{u_1,\cdots,u_{\left\lfloor\frac{n}{2}\right\rfloor}\}$. Let $H$ be a graph obtained from $A$ and $B$ by adding additional edges $e_i=u_iv_i$
for $i\in[\left\lfloor\frac{n}{2}\right\rfloor]$. Then $\delta(G)=\left\lfloor\frac{n}{2}\right\rfloor$. Because
$M=\{e_1,\cdots,e_{\left\lfloor\frac{n}{2}\right\rfloor}\}$ is a matching cut of $G$, by Lemma \ref{matching-cut}, $md(G)\geq2$.
\end{proof}

\section{Erd\H{o}s-Gallai-type problems}

Since for a connected graph $G$, we have $1\leq md(G)\leq n-1$, the Erd\H{o}s-Gallai-type problems for the monochromatic disconnection number
are stated as follows.

{\bf Problem A:} Given two positive integers $n$ and $r$ such that $1\leq r\leq n-1$, compute the minimum integer $f(n,r)$ such that
for any connected graph $G$ of order $n$, if $e(G)\geq f(n,r)$, then $md(G)\leq r$.

{\bf Problem B:} Given two positive integers $n$ and $r$ such that $1\leq r\leq n-1$, compute the maximum integer $g(n,r)$ such that
for any connected graph $G$ of order $n$, if $e(G)\leq f(n,r)$, then $md(G)\geq r$.

next we will consider the two problems separately in subsections.

\subsection{ Solution for Problem A}

In order to solve Problem A, we need the following lemmas.

\begin{lemma}\label{comb}
Let $G$ be a connected graph with $n$ vertices and $r$ blocks. Then $e(G)\leq {n-r+1\choose 2}+r-1$.
\end{lemma}
\begin{proof}
Let $H$ be a connected graph with $n$ vertices and $r$ blocks such that $e(H)$ is maximum.
We only need to prove $e(H)={n-r+1\choose 2}+r-1$. It is obvious that each block of $H$ is a complete graph.
In fact, the graph $H$ has $r-1$ trivial blocks $K_2$ and one block $K_{n-r+1}$, and then $e(H)={n-r+1\choose 2}+r-1$.
Otherwise, suppose $H$ has at least two non-trivial blocks $B_1$ and $B_2$ and $|B_1|\geq |B_2|$. Let $H'$ be a graph
obtained from $H$ by replacing $B_1$ by $K_{|B_1|+1}$ and replacing $B_2$ by $K_{|B_2|-1}$. Then $H'$ is a graph with
$n$ vertices, $r$ blocks and more edges, which contradicts that $e(H)$ is maximum.
\end{proof}

\begin{lemma}\label{md1}
Suppose $G$ is a graph with $n\geq 4$ and $e(G)\geq {n-1\choose 2}+2$. Then $md(G)=1$, and the lower bound for $e(G)$ is sharp.
\end{lemma}
\begin{proof}
The proof proceeds by induction on $n$. If $n=4$, then $G$ is either a $K_4$ or a $K_4^-$, and so $md(G)=1$. Let $G$ be a graph
with $n>4$. If $G$ is $K_n$, then $md(G)=1$. Otherwise there exists a vertex $v$ of $V(G)$ such that $d_G(v)\leq n-2$. Then $G'=G-v$ satisfies
$$e(G')=e(G)-d_G(v)\geq {n-1\choose 2}+2-(n-2)={n-2\choose 2}+2.$$
By induction, $md(G')=1$.

Because $e(G)\geq {n-1\choose 2}+2=e(K_n)-(n-3)$, $d_G(v)\geq2$, i.e., $v$ is not a pendent vertex. In fact, $v$ is not a cut-vertex,
for otherwise $G$ has at least $2$ blocks, and then $e(G)\leq {n-1\choose 2}+1$ by Lemma \ref{comb}, a contradiction.
Therefore $v$ is neither a pendent vertex nor a cut-vertex, and by Lemma \ref{secq}, $md(G')\geq md(G)$. So $md(G)=1$.

Let $H$ be a graph obtained by adding a pendent edge to a $K_{n-1}$. Then $e(H)={n-1\choose 2}+1$ and $md(H)=2$. This implies
that the bound is sharp.
\end{proof}

\begin{theorem}
Given two positive integers $n$ and $r$ with $1\leq r\leq n-1$,
$$ f(n,r)=\left\{
\begin{array}{lcl}
{n-r+1\choose 2}-n+2r+1&  & 1\leq r\leq n-2; \\
n-1&  &r=n-1.
\end{array}
\right. $$
\end{theorem}
\begin{proof}
Although the notation $f(n,r)$ has a special meaning in Problem $A$, for convenience, we just see it as function on the variables $n$ and $r$ in this proof.

If $n\leq 4$, it is easy to verify that the theorem holds. By Proposition \ref{block},
$f(n,n-1)=n-1$ is obvious. By Lemma \ref{md1}, the theorem holds when $r=1$. Therefore, we only need to show that $f(n,r)={n-r+1\choose 2}-n+2r+1$ when $n\geq5$ and $2\leq r\leq n-2$.

Let $G_1$ be a graph with $r-1$ trivial blocks and one non-trivial block $B$, where $|B|=n-r+1$ and $e(B)={n-r+1\choose 2}-n+r+2$. Then $e(B)={|B|-1\choose2}+2$,
and by Lemma \ref{md1}, $md(B)=1$. Therefore $md(G_1)=r$ by Proposition \ref{block}. Let $G_2$ be a graph with $r$ trivial blocks and one non-trivial block $K_{n-r}$.
Then $md(G_2)=r+1$. Because $e(G_1)=f(n,r)$ and $e(G_2)=f(n,r)-1$, we only need to show that $md(G)\leq r$ when $e(G)\geq f(n,r)$. In fact, since every graph with more
than $f(n,r)$ edges has a spanning subgraph with exactly $f(n,r)$ edges, by Lemma \ref{sub}, we only need to show that $md(G)\leq r$ when $e(G)= f(n,r)$.

Obviously, the result is true for $n\leq 4$. Suppose the result does not hold for all $n$. Let $n$ be the minimum integer such that there is a positive integer $r$
with $2\leq r\leq n-2$, the result is false for some connected graphs $G$ with $|G|=n$ and $e(G)= f(n,r)$.
We choose such a graph $G$ with $md(G)\geq r+1$ such that the number of blocks of $G$ is maximum.
Suppose $G$ has $t$ blocks $B_1,\cdots,B_t$. By Lemma \ref{comb}, $t\leq r$. Because $md(G)\geq r+1$, by Proposition \ref{block}, there is a block, say $B_1$, with $md(B_1)=k\geq2$.
Let $|B_1|=n_1$. We distinguish the following cases.

{\em Case 1.} $t\geq 2$.

Because $|B_1|=n_1<n$, $e(B_1)\leq f(n_1,k-1)-1={n_1-k+2\choose 2}-n_1+2(k-1)$. Let $T^k$ be a graph with $k-1$ trivial blocks and one block $K_{n_1-k+1}$, then $md(T^k)=k$
and $e(T^k)={n_1-k+1\choose 2}+k-1= f(n_1,k-1)-1\geq e(B_1)$. Let $G'$ be a graph obtained from $G$ by replacing $B_1$ by $T^k$ and let $G''$ be a connected spanning subgraph
of $G'$ with $f(n,r)$ edges. Then $G''$ is a graph with $|G''|=n$, $e(G'')= f(n,r)$ and $md(G'')\geq r+1$. However, the number of blocks of $G''$ is more than $t$, a contradiction.

{\em Case 2.} $t=1$.

Since $G$ has just one block, $G$ is $2$-connected. The average degree of $G$ is
$$\frac{2e(G)}{n}=\frac{ 2[{n-r+1\choose 2}-n+2r+1]}{n}=\frac{n^2-2nr+r^2-n+3r+2}{n}.$$
Since $G$ is $2$ connected, $md(G)=r\leq\left\lfloor\frac{n}{2}\right\rfloor$ by Theorem \ref{n-2}. Because $n\geq 5$ and $r\geq 2$, the difference between the average degree of $G$ and $n-r-1$ is
$$dif=\frac{2e(G)}{n}-(n-r-1)=\frac{r^2+3r+2}{n}-r.$$
Since $2\leq r\leq \left\lfloor\frac{n}{2}\right\rfloor$, if $n\geq8$, then $dif\leq0$; if $n=7$, then $dif<0$; if $n=6$, then $dif<1$; if $n=5$, then $dif<1$.
This implies that $G$ has a vertex $v$ with $d_G(v)\leq n-r-1$. Let $G'=G-v$. Then $G'$ is connected and $e(G')\geq e(G)-(n-r-1)=f(n-1,r)$.
Since $G$ is a minimum counterexample of the theorem and $|G'|=|G|-1$, $md(G')\leq r$. By Lemma \ref{secq}, $md(G)\leq md(G')\leq r$, which contradicts that $md(G)\geq r+1$.

According to above two cases, such a graph $G$ is not exists, and therefore the theorem holds.
\end{proof}

\subsection{Solution for Problem B}

To {\em contract} an edge $e$ of a graph $G$ is to delete the edge and then identify its ends,
and to contract an edge subset $X$ of a graph $G$ is to contract the edges of $X$ one by one.
The resulting graphs are denoted by $G/e$ and $G/X$, respectively.
To {\em subdivide} an edge of a graph is to insert a new vertex into the edge.
Let $v$ be a $2$-degree vertex of a graph $G$, and let $e_1=vv_1$ and $e_2=vv_2$ be two edges of
$G$ incident with $v$. The operation of {\em splitting off} the edges $e_1$ and $e_2$ from $v$
consists of deleting the vertex $v$ and its incident edges $e_1,e_2$ and then adding a new edge
joining $v_1$ and $v_2$.

\begin{claim} \label{long}
For a connected graph $G'$, let $c$ be a $2$-degree vertex of $G'$ and $e_1=ac$ and $e_2=bc$ be the two edges
incident with $c$. Let $G$ be a graph obtained from $G'$ by splitting off the $e_1$ and $e_2$ by a new edge $e$.
If $\Gamma'$ and $\Gamma$ are edge-colorings of $G'$ and $G$, respectively,
such that $\Gamma'(f)=\Gamma(f)$ when $f \in E(G'-v)$ and $\Gamma'(e_1)=\Gamma'(e_2)=\Gamma(e)$,
then $\Gamma'$ is an $MD$-coloring of $G'$ if and only if $\Gamma$ is an $MD$-coloring of $G$.
Furthermore, $md(G)\leq md(G')$.
\end{claim}
\begin{proof}
Since $G'$ is a connected graph, $G$ is also connected.
Let $E'_i$ and $E_i$ be the $i$-induced edge sets of $G'$ and $G$, respectively. Then $E_i=E'_i$ when $i\neq \Gamma(e)$
and $E_i=E'_i\cup e-(e_1\cup e_2)$ when $i=\Gamma(e)$. Furthermore, $V(G)=V(G')-c$ and $|\Gamma'(G')|=|\Gamma(G)|$.
The relationships between $G-E_i$ and $G'-E'_i$ are shown as follows.
\begin{enumerate}
\item If $i\neq \Gamma(e)$, then $E(G)-E_i$ is a graph obtained from $G'-E'_i$ by spitting off $e_1$ and $e_2$ from $c$;
\item if $i=\Gamma(e)$, then $G-E_i=(G'-E'_i)-c$.
\end{enumerate}

We prove the first result below, that is, $\Gamma'$ is an $MD$-coloring of $G'$ if and only if $\Gamma$ is an $MD$-coloring of $G$.
Suppose $\Gamma'$ is an $MD$-coloring of $G'$. Let $u,v$ be two vertices of $V(G)$. Since $u,v$ are also vertices of $V(G')$,
there is an $E'_i$ such that $u,v$ are in different components of $G'-E'_i$. According to the relationship between $G-E_i$ and $G'-E'_i$,
$u,v$ are also in different components of $G-E_i$. Therefore, $\Gamma$ is an $MD$-coloring of $G$.
Analogously, suppose $\Gamma$ is an $MD$-coloring of $G$. Let $u,v$ be two vertices of $V(G')$. If $u$ and $v$ are in $V(G')-c=V(G)$,
then there is an $E_i$ such that $u,v$ are in different components of $G-E_i$. According to the relationship between $G-E_i$ and $G'-E'_i$,
$u,v$ are also in different components of $G'-E'_i$; if one of the $u,v$ is $c$, since $c$ is an isolate vertex of $G'-E'_{\Gamma(e)}$,
$u,v$ are in different components of $G'-E'_{\Gamma(e)}$. Therefore, $\Gamma'$ is an $MD$-coloring of $G'$.

The second result can be derived from the first result directly.
Suppose the edge-coloring $\Gamma$ is an extremal $MD$-coloring of $G$. Then $\Gamma'$ is an $MD$-coloring of $G'$.
Since $|\Gamma|=|\Gamma'|$, we have $md(G)\leq md(G')$.
\end{proof}

\begin{lemma}\label{2-mat-cut}
Let $M$ be a minimal matching cut of $G$, and $G'$ be the underling graph of $G/M$. Then $md(G')\leq md(G)-1$.
\end{lemma}
\begin{proof}
The graph $G/M$ may have parallel edges but does not have loops. By Proposition \ref{parallel}, we only need to prove $md(G/M)\leq md(G)-1$.

Since $M$ is a minimal matching cut, $M$ is a bond of $G$. Then $G-M$ has two components,
say $D_1$ and $D_2$. We denote $M=\{e_1,\cdots,e_t\}$, where $e_i=a_ib_i$ and $a_i$ is in $D_1$ and $b_i$ is in $D_2$ for every $i\in[t]$.
Suppose the graph $G/M$ identifies the ends of $e_i$ into $c_i$. Let $A=\bigcup_{i\in[t]}(a_i\cup b_i)$ and let $f:V(G)\rightarrow V(G/M)$ be
a mapping such that $f(u)=u$ when $u\in V(G)-A$ and $f(u)=c_i$ when $u\in \{a_i,b_i\}$.

Let $\Gamma$ be an extremal $MD$-coloring of $G/M$ with $\Gamma=[md(G/M)]$ and let $E_i$ be the $i$-induced edge set of $G/M$.
Let $\Gamma'$ be an edge-coloring of $G$ such that $\Gamma(e)=\Gamma'(e)$ when $e\notin M$ and $\Gamma'(e)=md(G/M)+1$ when $e\in M$.

For any two vertices $u,v$ of $G$, if $f(u)$ and $f(v)$ are different vertices of $G/M$, then there is an $E_i$ such that $f(u)$ and $f(v)$
are in different components of $G/M-E_i$. Since $G-E_i$ is a graph obtained from $G/M-E_i$ by replacing each $c_i$ by $e_i$,
$u$ and $v$ are also in different components of $G-E_i$. If $f(u)=f(v)$, then $u=a_i$ and $v=b_i$ for some $i\in[t]$,
$u$ and $v$ are in different components of $G-M$. Therefore, $\Gamma'$ is an $MD$-coloring of $G$, and so
$md(G/M)=|\Gamma|=|\Gamma'|-1\leq md(G)-1$.
\end{proof}

The following are some definitions.

$\bullet$ A {\em semi-wheel} $SW(u;v_1v_2\cdots v_n)$ is a graph obtained by connecting $u$ to each vertex of the
path $P=v_1e_1v_2e_2\cdots e_{n-1}v_n$.

$\bullet$ For $n\geq 3$, let $D_n$ be a graph obtained from $SW(u;v_1v_2\cdots v_n)$ by subdividing $uv_2,uv_3,\cdots,uv_{n-1}$.
We call $uv_1$ and $uv_n$ the {\em verges} of $D_n$.

$\bullet$ For $n\geq 4$, let $F_n$ be a graph obtained from $SW(u;v_1v_2\cdots v_n)$ by subdividing $uv_2,uv_3,\cdots,uv_{n-2}$.

$\bullet$ We construct a graph $H_n$ as follows:
$$ H_n=\left\{
\begin{array}{lcl}
K_n&  & n=1,2,3; \\
K^-_4&  & n=4; \\
D_{\frac{n+1}{2}} &   & n \mbox{ is odd and } n\geq5;\\
F_{\frac{n+2}{2}}&  &n \mbox{ is even and } n\geq6.
\end{array}
\right. $$

$\bullet$ Suppose $v_1$ and $v_2$ are pendent vertices of a path $P$ and $u_1,u_2$ are two different vertices of a graph $G$,
and $V(P)\cap V(G)=\emptyset$. We use $I(P,G)$ to denote a graph obtained by identifying $u_i$ of $G$ and $v_i$ of $P$, respectively, for $i\in[2]$.

$\bullet$ Let $n$ and $r$ be two integers with $3\leq r\leq \left\lfloor\frac{n}{2}\right\rfloor$. We construct a graph $H_{n,r}$ below.
If $n$ is even and $r<\frac{n}{2}$, then $H_{n,r}=I(P,H_{n-2r+1})$ where $P$ is a $2r$-path; if $n$ is even and $r=\frac{n}{2}$,
then $H_{n,r}=C_n$; if $n$ is odd, then $H_{n,r}=I(P,H_{n-2r+2})$ where $P$ is a $(2r-1)$-path.

\begin{remark} From the above definitions, we have $e(H_n)=\left\lceil\frac{3}{2}(n-1)\right\rceil$ when $n\geq 3$.
For $n\geq 6$, $e(H_{n,r})=\left\lceil\frac{3}{2}(n-2r)\right\rceil+2r=\frac{3n}{2}-r$ when $n$ is even and $e(H_{n,r})=\left\lceil\frac{3}{2}(n-2r+1)\right\rceil+2r-1=\frac{3n+1}{2}-r$ when $n$ is odd.
For convenience of discussion, if $n\geq 6$ and $3\leq r\leq \left\lfloor\frac{n}{2}\right\rfloor$, then we denote $\mu_{n,r}=\left\lceil\frac{3}{2}(n-2r)\right\rceil+2r$ when $n$ is even
and $\mu_{n,r}=\left\lceil\frac{3}{2}(n-2r+1)\right\rceil+2r-1$ when $n$ is odd, i.e., $e(H_{n,r})=\mu_{n,r}$.
\end{remark}

The following is the proof of $md(H_n)=1$ for $n\geq 2$. The proof uses an obvious conclusion that any
$MD$-coloring of a $4$-cycle or a $5$-cycle is either trivial or assigning colors $1$ and $2$ alternately to its edges.
Therefore, there are two adjacent edges of the $5$-cycle receiving a same color when the $MD$-coloring is non-trivial.
\begin{lemma}
$md(H_n)=1$ for $n\geq 2$.
\end{lemma}
\begin{proof}
Because $H_2=K_2, H_3=K_3, H_4=K_4^-$ and $H_5=K_{2,3}$, by Lemma \ref{kst} we have $md(H_n)=1$ for $2\leq n\leq 5$.
We proceeds the proof by induction on $n$. The lemma holds when $n\leq 5$. Now suppose $n\geq 6$.

If $n$ is even, then $H_n=H_{n-1}\cup K_3$ and the intersecting edge of $H_{n-1}$ and  $K_3$ is a verge of $H_{n-1}$.
Since $md(H_{n-1})=md(K_3)=1$, by Lemma \ref{subb} we have $md(H_n)=1$. Therefore, we only need to show that $md(H_n)=1$ when $n$ is odd.
Let $n=2k-1$ and $k\geq 3$.

Let $H_n=H_{2k-1}$ be a graph obtained by inserting new vertices $w_2,\cdots,w_{k-1}$ to $uv_2,\cdots,uv_{k-1}$ of $SW(u;v_1v_2\cdots v_k)$, respectively.
Here $e_i=v_iv_{i+1}$ for $i\in [k-1]$ and $P=v_1e_1\cdots e_{k-1}v_k$ is a path.

We proceeds the proof by contradiction. Suppose $md(H_{2k-1})\geq 2$. Then by Lemma \ref{leqk}, there exists an $MD$-coloring $\Gamma$ of $H_{2k-1}$
such that $|\Gamma|=2$, i.e., every edge of $H_{2k-1}$ is either colored by $1$ or colored by $2$. We distinguish the following two cases.

{\em Case 1.} There exist adjacent edges $e_i$ and $e_{i+1}$ of $P$ such that $\Gamma(e_i)=\Gamma(e_{i+1})$.

Let $H=H_{2k-1}-w_{i+1}$.
Then $\Gamma$ is an $MD$-coloring restricted on $H$. Furthermore, $|\Gamma(H)|=2$. Otherwise suppose all edges of $H$ are colored by $1$.
Since $|\Gamma|=2$, at least one of $e_1$ and $e_2$ is colored by $2$ under $\Gamma$. Since $e_1$ and $e_2$ are in the $5$-cycle
$C=H_{2k-1}[u,w_i,v_i,v_{i+1},w_{i+1}]$, $\Gamma$ is not an $MD$-coloring restricted on $C$, a contradiction.

Let $H'$ be a graph obtained from $H$ by splitting off $e_i$ and $e_{i+1}$ from $v_{i+1}$.
By Claim \ref{long}, there is an $MD$-coloring $\Gamma'$ of $H'$ such that $|\Gamma'|=2$. However, $H'=H_{2k-3}$, and by induction, $md(H')=1$,
a contradiction.

{\em Case 2.} Assigning colors $1$ and $2$ alternately on $P$, i.e., $\Gamma(e_j)=1$ when $j$ is odd and $\Gamma(e_j)=2$ when $j$ is even.

\begin{figure}[h]
    \centering
    \includegraphics[width=300pt]{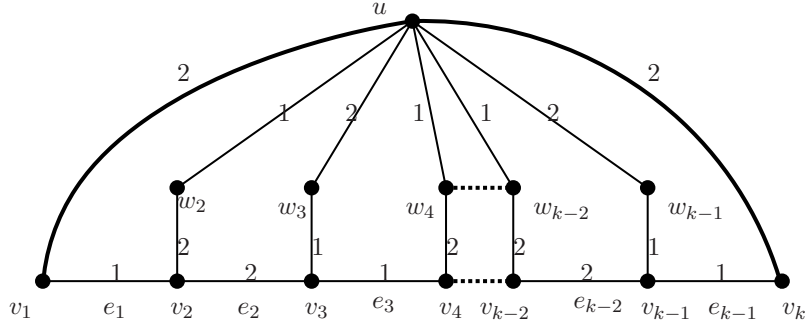}\\
    \caption{The graph for Case $2$ with $k$ is even.} \label{Case2}
\end{figure}

If $\Gamma(uv_1)=\Gamma(e_1)=1$, then $\Gamma$ is a trivial $MD$-coloring restricted on the $4$-cycle $H_{2k-1}[u,v_1,v_2,w_2]$, and so $\Gamma(uw_2)=\Gamma(w_2v_2)=1$.
Let $H$ be a graph obtained from $H_{2k-1}$ by splitting off $uw_2$ and $w_2v_2$ from $w_2$. Then by Claim \ref{long}, there is an $MD$-coloring $\Gamma'$ of $H'$ such that
$|\Gamma'|=2$. However, $H'=H_{2k-2}$, and by induction, $md(H')=1$, a contradiction.

If $\Gamma(uv_1)\neq\Gamma(e_1)$, then each $5$-cycle $C_i=H_{2k-1}[u,w_i,v_i,v_{i+1},w_{i+1}]$ is colored non-trivially under $\Gamma$.
Furthermore, $\Gamma(w_iv_i)=\Gamma(e_i)$ for $i=2,\cdots,k-1$. This implies that $\Gamma(w_{k-2}v_{k-2})=\Gamma(e_{k-2})=\Gamma(uw_{k-1})$.
Since $\Gamma(e_{k-2})\neq\Gamma(e_{k-1})$, we have $\Gamma(uw_{k-1})\neq\Gamma(e_{k-1})$, which contradicts that $\Gamma$ is an $MD$-coloring restricted
on the $4$-cycle $H_{2k-1}[u,w_{k-1},v_{k-1},v_k]$.

According to the above two cases, one has $md(H_{2k-1})=1$. The proof is thus complete. 
\end{proof}

\begin{lemma}
If $3\leq r\leq \left\lfloor\frac{n}{2}\right\rfloor$ and $n\geq 6$, then $md(H_{n,r})=r$.
\end{lemma}
\begin{proof}
Let $Q_1=v_1e_1v_2e_2\cdots v_{2r}e_{2r}v_{2r+1}$ and $Q_2=v_1e_1v_2e_2\cdots v_{2r-1}e_{2r-1}v_{2r}$. Let $R_1=H_{n-2r+1}$ and $R_2=H_{n-2r+2}$. We will construct $H_{n,r}$ below.
If $n$ is even and $r=\frac{n}{2}$, then $H_{n,r}=C_n$;
if $n$ is even and $3\leq r<\frac{n}{2}$, then $H_{n,r}=I(Q_1,R_1)$; if $n$ is odd, then $H_{n,r}=I(Q_2,R_2)$.

{\em Case 1.} $n$ is even and $r=\frac{n}{2}$.

Since $H_{n,r}=C_n$, by Proposition \ref{block}, $md(H_{n,r})=r$ holds.

{\em Case 2.} $n$ is even and $3\leq r<\frac{n}{2}$.

Color $e_i$ by $j\in[r]$ if $i\equiv j\pmod r$ and color the edges of $R_1$ by $1$. It is easy to verify that the edge-coloring
is an $MD$-coloring of $H_{n,r}$. Therefore, $md(H_{n,r})\geq r$. Since every edge of $H_{n,r}$ is in some cycles,
every color of an extremal $MD$-coloring of $H_{n,r}$ is used on at least two edges. Furthermore, since $md(R_1)=1$,
all edges of $R_1$ are colored the same under the extremal $MD$-coloring. Therefore, there are at most $r$ colors in
the extremal $MD$-coloring, and so $md(H_{n,r})\leq r$. Thus, $md(H_{n,r})=r$.

{\em Case 3.} $n$ is odd and $3\leq r\leq\frac{n}{2}$.

Color $e_i$ by $j\in[r]$ if $i\equiv j\pmod r$ and color the edges of $R_2$ by $r$. It is obvious that the edge-coloring
of $H_{n,r}$ is an $MD$-coloring. Therefore, $md(H_{n,r})\geq r$. As discussed in Case $2$, since every color of an extremal
$MD$-coloring of $H_{n,r}$ is used on at least two edges and since $md(R_2)=1$, we have $md(H_{n,r})\leq r$. Thus, $md(H_{n,r})=r$.
\end{proof}

\begin{lemma}\label{2-n}
For $n\geq 4$, $g(n,2)=\left\lceil\frac{3}{2}(n-1)\right\rceil-1$. For $n\geq 6$, $g(n,\left\lfloor\frac{n}{2}\right\rfloor)=\mu_{n,\left\lfloor\frac{n}{2}\right\rfloor}$.
\end{lemma}
\begin{proof}
For $n\geq 4$, since $md(H_n)=1$ and $e(H_n)\leq\left\lceil\frac{3}{2}(n-1)\right\rceil$, we have $g(n,2)\leq\left\lceil\frac{3}{2}(n-1)\right\rceil-1$.
By Theorem \ref{mat-i}, $G$ has a matching cut when $e(G)\leq\left\lceil\frac{3}{2}(n-1)\right\rceil-1$, and by Lemma \ref{matching-cut}, we have $md(G)\geq 2$.
Therefore, $g(n,2)=\left\lceil\frac{3}{2}(n-1)\right\rceil-1$.

If $n\geq 6$ and $n$ is even, $g(n,\frac{n}{2})\leq \mu_{n,\frac{n}{2}} =n$ by Corollary \ref{lleq}.
Since any connected graph $G$ with $e(G)\leq n$ is either a tree or a unicyclic graph, we have $md(G)\geq \frac{n}{2}$ by Proposition \ref{block}.
Therefore, $g(n,\frac{n}{2})= n$ when $n$ is even.

If $n\geq 7$ and $e(G)=n+1$, we first show that $G$ has a minimal matching cut $M$ such that $|M|\leq 2$.
If $G$ has a cut-edge, then $|M|=1$. Otherwise $G$ has at most two non-trivial blocks. Furthermore,
either $G$ has exactly two $3$-degree vertices and the other vertices are $2$-degree vertices,
or $G$ has one $4$-degree vertex and the other vertices are $2$-degree vertices,
and both cases imply that there are two adjacent $2$-degree vertices, say $u$ and $v$. Let $e_1=xu$, $e_2=uv$ and $e_3=vy$, where $x\neq v$ and $y\neq u$.
If $x\neq y$, $M=\{e_1,e_3\}$; if $x=y$, one block of $G$ is $K_3$ and the other block is an $(n-2)$-cycle. Since $n\geq 7$, the $(n-2)$-cycle has a matching cut $M$
and $|M|=2$. $M$ is also a matching cut of $G$.

Now we show that if $n$ is odd and $n\geq 7$, $g(n,\left\lfloor\frac{n}{2}\right\rfloor) =n+1$. By Corollary \ref{lleq},
$g(n,\left\lfloor\frac{n}{2}\right\rfloor)\leq \mu_{n,\left\lfloor\frac{n}{2}\right\rfloor} =n+1$. In order to show 
$g(n,\left\lfloor\frac{n}{2}\right\rfloor)= \mu_{n,\left\lfloor\frac{n}{2}\right\rfloor}=n+1$,
we need to prove that any graph $G$ with $|G|=n$ and $e(G)\leq n+1$ has $md(G)\geq\left\lfloor\frac{n}{2}\right\rfloor$. Let $G$ be a connected graph with $|G|\geq 7$ and $e(G)\leq n+1$.
Then $G$ has a minimal matching cut $M$ such that $|M|\leq 2$. Let $G'$ be the underling simple graph of $G/M$. By Lemma \ref{2-mat-cut},
$md(G')\leq md(G)-1$. So, we only need to show $md(G')\geq \left\lfloor\frac{n}{2}\right\rfloor-1$.

If $|M|=1$, since $|G'|$ is even and $e(G')= |G'|+1=\mu_{n-1,\left\lfloor\frac{n-1}{2}\right\rfloor-1}$, we have 
$md(G')\geq \left\lfloor\frac{n-1}{2}\right\rfloor-1=\left\lfloor\frac{n}{2}\right\rfloor-1$.

If $|M|=2$, there are two cases to consider.

{\em Case 1.} $n=7$.

Then $|G/M|=5$ and $e(G/M)\leq6$. It is easy to verify that $G/M=H_5$ is the only such graph with $md(G/M)=1$. If $G/M\neq H_5$,
then $md(G/M)=2=\left\lfloor\frac{n}{2}\right\rfloor-1$; if $G/M=H_5$, then the graph $G$ and one of its $MD$-colorings are shown as in Figure \ref{H-5}, and so $md(G)\geq3$.

\begin{figure}[h]
    \centering
    \includegraphics[width=100pt]{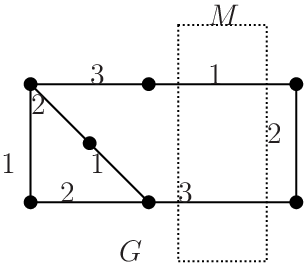}\\
    \caption{The graph $G$ that satisfies $G/M=H_5$, and an $MD$-coloring of $G$.} \label{H-5}
\end{figure}

{\em Case 2.} $n\geq 9$. Since $|G'|=n-2$ is odd and $e(G')\leq |G'|+1=\mu_{n-2,\left\lfloor\frac{n-2}{2}\right\rfloor}$,
by induction, $md(G')\geq \left\lfloor\frac{n-2}{2}\right\rfloor=\left\lfloor\frac{n}{2}\right\rfloor-1$.
\end{proof}

\begin{lemma}\label{E-G-gleq}
If $2\leq r-1<r\leq \frac{n}{2}$, then $g(n,r)\leq g(n,r-1)$.
\end{lemma}
\begin{proof}
For any graph $G$ with $v(G)=n$ and $e(G)\leq g(n,r)$, $md(G)\geq r$.
This also implies $md(G)\geq r-1$,
i.e., $g(n,r)\leq g(n,r-1)$.
\end{proof}

\begin{lemma}\label{E-G-gfunodd}
If $l\geq 3$ and $n\geq 7$ is odd, then $g(n,l)=\mu_{n,l}=\frac{3n+1}{2}-l$.
\end{lemma}
\begin{proof}
If $n$ is odd, then by Lemma \ref{2-n}, $g(n,2)=\left\lceil\frac{3(n-1)}{2}\right\rceil-1$ and 
$g(n,\frac{n-1}{2})=\mu_{n,\frac{n-1}{2}}=n+1$. Since $g(n,2)-g(n,\frac{n-1}{2})=\frac{n-1}{2}-3$ and $g(n,l-1)\geq g(n,l)$,
there is the maximum integer integer $3\leq r\leq \frac{n-1}{2}$ such that $g(n,r-1)=g(n,r)$.

\begin{claim}
$g(n,l)=\frac{3n+1}{2}-l$ for $r\leq l\leq \frac{n-1}{2}$.
\end{claim}
\begin{proof}
If $r=\frac{n-1}{2}$, then by Lemma \ref{2-n}, the result holds.
Thus, suppose $r<\frac{n-1}{2}$.
Since $g(n,l)\leq g(n,l-1)$ and $r$ is a maximum integer such that $g(n,r-1)=g(n,r)$, 
we have $g(n,l+1)<g(n,l)$ for $r\leq l\leq \frac{n-1}{2}-1$.
Suppose the claim does not hold.
Then let $p$ be the maximum integer such that $g(n,p)\leq g(n,p-1)-2$.
Thus, $g(n,l+1)=g(n,l)-1$ holds for $p\leq l\leq \frac{n-1}{2}-1$.
Since $g(n,\frac{n-1}{2})=n+1=\frac{3n+1}{2}-\frac{n-1}{2}$, $g(n,l)=\frac{3n+1}{2}-l$ holds for $p\leq l\leq \frac{n-1}{2}$.
Thus, $g(n,p-1)\geq \frac{3n+1}{2}-p+2$.
If $p-2\geq 3$, then since $e(H_{n,p-2})=\frac{3n+1}{2}-p+2\leq g(n,p-1)$ and $md(H_{n,p-2})=p-2<p-1$, 
this yields a contradiction.
If $p-2\leq 2$, then $g(n,p-1)\geq \frac{3n+1}{2}-2=\left\lceil\frac{3(n-1)}{2}\right\rceil=e(H_n)$.
However, $md(H_n)=1<p-1$, a contradiction.
Thus, $g(n,l+1)=g(n,l)-1$ holds for $r\leq l\leq \frac{n-1}{2}$.
Since $g(n,\frac{n-1}{2})=n+1=\frac{3n+1}{2}-\frac{n-1}{2}$, the result holds.
\end{proof}

Therefore, $g(n,r-1)=g(n,r)=\frac{3n+1}{2}-r$.
For any graph $G$ with $v(G)=n$ and $e(G)\leq \frac{3n+1}{2}-r$, $md(G)\geq r$.
If $r\geq 4$, then since $e(H_{n,r-1})=\frac{3n+1}{2}-r=g(n,r)$ and $md(H_{n,r-1})=r-1<r$, 
this yields a contradiction.
Thus, $r=3$. Therefore, $g(n,l)=\frac{3n+1}{2}-l$ for $3\leq l\leq \frac{n-1}{2}$ and
$g(n,2)=g(n,3)=\frac{3n+1}{2}-3$.
\end{proof}

\begin{lemma}\label{E-G-gfuneven}
If $\left\lfloor\frac{n}{2}\right\rfloor-1\geq r\geq 3$ and $n\geq 8$ is even, then $g(n,r)=\mu_{n,r}$.
\end{lemma}
\begin{proof}
Suppose $G$ is a graph with $e(G)\leq \frac{3n}{2}-r$.
Since $\frac{2e(G)}{n}<3$, there is a vertex $v$ with degree two or one. If $d_G(G)=1$, let $G'=G-v$, and then $md(G')=md(G)-1$;
if $d_G(v)=2$, then let $G'$ be a graph obtained from $G$ by splitting off the two edges incident with $v$.
By Claim \ref{long}, $md(G')\leq md(G)$. Therefore, $md(G')\leq md(G)$ and $e(G')=e(G)-1=\mu_{n-1,r}$ in both cases.
Since $r\leq \left\lfloor\frac{n}{2}\right\rfloor-1$, we also have $r\leq \left\lfloor\frac{n-1}{2}\right\rfloor$. 
Since $|G'|=n-1$ is odd and $e(G')=\mu_{n-1,r}$, we have $md(G')\geq r$. Therefore, $md(G)\geq r$.
\end{proof}

\begin{theorem}
For $n\geq 2$ and $1\leq r\leq n-1,$
$$ g(n,r)=\left\{
\begin{array}{lcl}
\frac{n(n-1)}{2}&  & r=1; \\
\left\lceil\frac{3}{2}(n-1)\right\rceil-1&  & r=2; \\
\frac{3n+1}{2}-r&  & n\geq 7\mbox{ is odd and }3\leq r\leq\left\lfloor\frac{n}{2}\right\rfloor; \\
\frac{3n}{2}-r&  & n\geq 6\mbox{ is even  and }3\leq r\leq\left\lfloor\frac{n}{2}\right\rfloor; \\
n-1 &   &\left\lfloor\frac{n}{2}\right\rfloor+1\leq r\leq n-1,
\end{array}
\right. $$
\end{theorem}
\begin{proof}
It is easy to verify that $g(n,1)={n\choose 2}$ and $g(n,r)=n-1$ when $n-1\geq r\geq \left\lfloor\frac{n}{2}\right\rfloor+1$.
By Lemma \ref{2-n}, $g(n,2)=\left\lceil\frac{3}{2}(n-1)\right\rceil-1$ when $n\geq 4$.

If $3\leq r\leq \left\lfloor\frac{n}{2}\right\rfloor$ and $n\geq 7$ is odd, then by Lemma \ref{E-G-gfunodd}, $g(n,r)=\frac{3n+1}{2}-r$.
If $3\leq r\leq \frac{n}{2}-1$ and $n\geq 8$ is even,
then by Lemma \ref{E-G-gfuneven}, $g(n,r)=\frac{3n}{2}-r$.
If $n\geq 6$ is even and $r=\frac{n}{2}$, then by Lemma \ref{2-n}, $g(n,r)=\frac{3n}{2}-r$.
\end{proof}

\section{Results for graph products}

Since an $MD$-coloring of a $4$-cycle is either trivial or assigning $1$ and $2$ alternately to its edges,
the opposite edges of a $4$-cycle are colored the same under its every $MD$-coloring.

\begin{theorem}
For two connected graphs $G$ and $H$, $md(G\Box H)=md(G)+md(H)$.
\end{theorem}
\begin{proof}
Let $|G|=n_1$ and $|H|=n_2$. Let $V(G)=\{u_1,\cdots,u_{n_1}\}$ and $V(H)=\{v_1,\cdots,v_{n_2}\}$. For an edge $e=u_iu_j$ of $G$ and an edge $f=v_sv_t$ of $H$, let
$$S_e=\{((u_i,v_r),(u_j,v_r)):~r\in[n_2]\}\mbox{ and }S_f=\{((u_r,v_s),(u_r,v_t)):~r\in[n_1]\}.$$
It is obvious that every edge of $G\Box H$ is in a unique $S_e$, where $e$ is either in $E(G)$ or in $E(H)$. Therefore, $\bigcup_{e\in E(G)\cup E(H)}S_e=E(G\Box H)$.

Let $\Gamma$ be an extremal $MD$-coloring of $G\Box H$. Then we have the following result.
\begin{claim}\label{se}
$|\Gamma(S_e)|=1$ for every $e\in E(G)\cup E(H)$.
\end{claim}
\begin{proof}
Without loss of generality, let $e=u_1u_2$ be an edge of $G$. For any two edges $h_1=((u_1,v_i),(u_2,v_i))$ and $h_2=((u_1,v_j),(u_2,v_j))$ of $S_e$, there is a $v_iv_j$-path $P$ of $H$. W.l.o.g., let $v_i=v_1$ and $P=v_1f_1v_2f_2\cdots v_{j-1}f_{j-1}v_j$. Then $L=e\Box P$ is a subgraph of $G\Box H$. Because $e\Box f_r$ is a $4$-cycle for $r\in[j-1]$, and $((u_1,v_r),(u_2,v_r))$ and $((u_1,v_{r+1}),(u_2,v_{r+1}))$ are opposite edges of $e\Box f_r$, $((u_1,v_r),(u_2,v_r))$ and $((u_1,v_{r+1}),(u_2,v_{r+1}))$ are colored the same under $\Gamma$. Therefore, $h_1$ and $h_2$ are colored the same under $\Gamma$.
\end{proof}

Because $u_1\Box H$ and $G\Box v_1$ are subgraphs of $G\Box H$, by Proposition \ref{key}, $\Gamma$ is an $MD$-coloring restricted on $G\Box v_1$ and $u_1\Box H$. Since $G\cong G\Box v_1$ and $H\cong u_1\Box H$, $|\Gamma(G\Box v_1)|\leq md(G)$ and $|\Gamma(u_1\Box H)|\leq md(H)$.
Now we choose an edge $h$ of $G\Box H$ arbitrarily. Without loss of generality, suppose $h=((u_i,v_l),(u_j,v_l))$ (or $h=((u_r,v_s),(u_r,v_t))$). Then by Claim \ref{se}, there is an edge $e=((u_i,v_1),(u_j,v_1))$ of $G\Box v_1$ (or an edge $e=((u_1,v_s),(u_1,v_t))$ of $u_1\Box H$), such that $\Gamma(h)= \Gamma(e)$. This implies that $\Gamma(G\Box v_1)\cup \Gamma(u_1\Box v_1)= \Gamma$. Since $\Gamma$ is an extremal $MD$-coloring of $G\Box H$, $md(G\Box H)=|\Gamma|\leq md(G)+md(H)$.

We need to prove $md(G\Box H)\geq md(G)+md(H)$ below. Let $\Gamma_1$ be an extremal $MD$-coloring of $G$ and $\Gamma_2$ be an extremal $MD$-coloring of $H$ and $\Gamma_1\cap \Gamma_2=\emptyset$.
Since every edge $h$ of $G\Box H$ is in a unique $S_e$, where $e$ is either in $E(G)$ or $E(H)$, we construct an edge-coloring $\Gamma$ of $G\Box H$
such that $\Gamma(h)=\Gamma_1(e)$ when $e\in E(G)$ and $\Gamma(h)=\Gamma_2(e)$ when $e\in E(H)$. Since $|\Gamma|=|\Gamma_1|+|\Gamma_2|=md(G)+md(H)$,
in order to prove $md(G\Box H)\geq md(G)+md(H)$, we only need to prove that $\Gamma$ is an $MD$-coloring of $G\Box H$.

We need to prove that there is a monochromatic cut between any two different vertices of $G\Box H$. We set the two different vertices and denote them by $w_0=(u_i,v_s)$ and $w_r=(u_j,v_t)$, here either $u_i\neq u_j$
or $v_s\neq v_t$, say $v_s\neq v_t$. Since $\Gamma_2$ is an extremal $MD$-coloring of $H$, there is a monochromatic $u_sv_t$-cut of $H$, and we suppose that the color of the monochromatic $u_sv_t$-cut is $c$. If any $w_0w_r$-path of $G\Box H$ has an edge that is colored by $c$ under $\Gamma$, then the set of these edges is a monochromatic $w_0w_r$-cut of $G\Box H$ under $\Gamma$. We will show the existence below.

Let $P=w_0h_0w_1h_1\cdots w_{r-1}h_{r-1}w_r$ be a $w_0w_r$-path of $G\Box H$. Here $h_i=w_iw_{i+1}$ is an edge of $G\Box H$. For convenience, we denote $w_k$ by $(u_k,v_k)$ for $0\leq k\leq r$, and then $i=s=0$ and $j=t=r$. Because $h_k=w_kw_{k+1}=((u_k,v_k),(u_{k+1},v_{k+1}))$ is an edge of $G\Box H$, either $v_kv_{k+1}$ is an edge of $H$ or $v_k=v_{k+1}$.
Therefore, $L=v_sv_1\cdots v_{r-1}v_t$ is a $v_sv_t$-walk of $H$ (it may have $v_k=v_{k+1}$ for some $0\leq k\leq r-1$). Then $L$ contains a $v_sv_t$-path $L'$ of $H$. This implies that there is an edge of $L'$, which is also an edge of $L$, is colored by $c$. Suppose the edge is $e=v_lv_{l+1}$. Then $h_l=((u_l,v_l),(u_{l+1},v_{l+1}))$ is an edge of $P$ colored by $c$. This implies that any $w_0w_r$-path of $G\Box H$ has an edge that is colored by $c$ under $\Gamma$.

Since the $w_0w_r$-path $P$ is chosen arbitrarily, there is a
monochromatic $w_0w_r$-cut of $G\Box H$ under $\Gamma$, and since the vertices $w_0$ and $w_r$ are chosen arbitrarily, $\Gamma$ is an $MD$-coloring of $G\Box H$.
\end{proof}

Because any three graphs $G_1, G_2$ and $G_3$ satisfy $G_1\Box G_2\Box G_3=(G_1\Box G_2)\Box G_3$, the following result is obvious.
\begin{corollary}
For $k$ connected graphs $G_1,\cdots,G_k$, $md(G_1\Box\cdots\Box G_k)=\sum_{i\in[k]}md(G_i)$.
\end{corollary}

\begin{lemma}\label{2p}
If $m\geq 1$ and $n\geq 1$, then $P_m\boxtimes P_n$ is a closure.
\end{lemma}
\begin{proof}
The proof is by induction on $m+n$. It is easy to verify that $P_1\boxtimes P_1=K_4$, and so the result holds for $m+n=2$. Suppose $m+n>2$ and $m\geq 2$.
Let $P_m=u_0e_1u_1e_2\cdots u_{m-1}e_mu_m$ and $P_n=v_0f_1v_1f_2\cdots v_{n-1}f_nv_n$. Let $P'=P_m-e_m$, and by induction, both $P'\boxtimes P_n$ and $e_m\boxtimes P_n$ are closures.
Since $h=((u_{m-1},v_0),(u_{m-1},v_1))$ is a common edge of $P'\boxtimes P_n$ and $e_m\boxtimes P_n$, $P_m\boxtimes P_n$ is a closure.
\end{proof}

\begin{theorem}\label{K4}
For two connected graphs $G$ and $H$ with $|G|\geq 2$ and $|H|\geq 2$, $md(G\boxtimes H)=1$.
\end{theorem}
\begin{proof}
By Lemma \ref{closure}, if we prove $G\boxtimes H$ is a closure, then we are done.
Let $h_1=((x_1,y_1),(x_2,y_2))$ and $h_2=((a_1,b_1),(a_2,b_2))$ be two distinct edges of $G\boxtimes H$.
Let $e_1=x_1x_2$, $e_2=a_1a_2$, $f_1=y_1y_2$ and $f_2=b_1b_2$. Then $e_i$ (or $f_i$) is either an edge or a
vertex of $G$ (or $H$) for $i=1,2$. Therefore, there is a path $P'$ of $G$ connects $e_1$ and $e_2$, that is,
$e_1$ is either a pendent edge of $P'$ if $e_1$ is an edge, or a pendent vertex of $P'$ if $e_1$ is a vertex,
and so is $e_2$. Analogously, there is a path $P''$ of $H$ connects $f_1$ and $f_2$.
Furthermore, at least one of $e_1$ and $f_1$ is an edge, and at least one of $e_2$ and $f_2$ is an edge.

{\em Case 1.} None of $P'$ and $P''$ is a single vertex.

Since at least one of $e_1$ and $f_1$ is an edge, and at least one of $e_2$ and $f_2$ is an edge, without loss of generality,
we assume $e_1$ and $f_2$ are edges. Then $h_1\in E(e_1\boxtimes f_1)$ and $h_2\in E(e_2\boxtimes f_2)$.
Since both $e_1\boxtimes f_1$ and $e_2\boxtimes f_2$ are subgraphs of $P'\boxtimes P''$, both $h_1$ and $h_2$ are in $P'\boxtimes P''$.
By Lemma \ref{2p}, $P'\boxtimes P''$ is a closure, and then $h_1\theta h_2$ is in $P'\boxtimes P''$. Therefore, $h_1\theta h_2$ is also in $G\boxtimes H$.

{\em Case 2.} One of $P'$ and $P''$ is a single vertex, say $P'$.

Since at least one of $e_1$ and $f_1$ is an edge, and at least one of $e_2$ and $f_2$ is an edge, and since
$e_1=e_2$ is a vertex of $G$, both $f_1$ and $f_2$ are edges of $H$. Since $|G|\geq2$, there is an edge of $G$, say $e$,
incident with $e_1$. It is easy to verify that both $h_1$ and $h_2$ are in $e\boxtimes P''$. Since $e\boxtimes P''$ is a closure by Lemma \ref{2p},
$h_1\theta h_2$ in $e\boxtimes P''$. Therefore, $h_1\theta h_2$ is also in $G\boxtimes H$.
\end{proof}

Because $G\boxtimes H$ is a connected spanning subgraph of $G\circ H$ by Proposition \ref{S-L}, by Lemma \ref{sub}, the following result is obvious.

\begin{theorem}
If $G$ and $H$ are connected graphs with $|G|\geq 2$ and $|H|\geq 2$, then $md(G\circ H)=1$.
\end{theorem}

\begin{lemma}\label{KP}
$md(K_2\ast K_n)=md(P_3\ast K_3)=1$ where $n\geq 5$.
\end{lemma}
\begin{proof}
We first show that $md(K_2\ast K_n)=1$ for $n\geq5$. Let $V(K_2)=\{x_1,x_2\}$ and $V(K_n)=\{y_1,\cdots,y_n\}$. We construct a bipartite graph $G_{2,n}$ with bipartition
$S_1=\{v_1^1,v_2^1,\cdots,v_n^1\}$ and $S_2=\{v_1^2,v_2^2,\cdots,v_n^2\}$, and $v_i^s$ connects $v_j^t$ if and only if $i\neq j$ and $s\neq t$. Then $K_2\ast K_n\cong G_{2,n}$,
this is because there is a bijection $f$ between $V(K_2)\times V(K_n)$ and $V(G_{2,n})$, such that $f(x_i,y_j)=v_j^i$,
and then $((x_i,y_j),(x_s,y_t))$ is an edge of $K_2\ast K_n$ if and only if $v_j^iv_t^s$ is an edge of $G_{2,n}$.
Therefore, by Lemma \ref{closure}, we only need to prove that $G_{2,n}$ is a closure when $n\geq5$.

Let $e=v_i^1v_j^2$ and $f=v_s^1v_t^2$ be two edges of $G_{2,n}$. Then $i\neq j$ and $s\neq t$. Let $A=\{i,j,s,t\}$.

If $|A|=4$, since $n\geq5$, there is an integer $w\in[n]$ such that $w\notin A$. Then $i,j,s,t,w$ are pairwise different, and so $G_1=G_{2,n}[v_i^1,v_j^2,v_s^1,v_t^2,v_w^2]\cong K_{2,3}$.
Therefore, $e\theta f$.

If $|A|=3$, then if $e$ and $f$ have no common vertex, for convenience, let $i=t=1$, $j=2$ and $s=3$. Then $G_1=G_{2,n}[v_i^1,v_j^2,v_3^1, v_4^2,v_5^1]\cong K_{2,3}$
and $G_2=G_{2,n}[v_s^1,v_t^2,v_2^1, v_4^2,v_5^1]\cong K_{2,3}$. Since $e\in E(G_1)$, $f\in E(G_2)$ and $v_4^2v_5^1\in E(G_1)\cap E(G_2)$, $e\theta f$.
If $e$ and $f$ have a common vertex, for convenience, let $i=s=1$, $j=2$ and $t=3$. Then $G'_1=G_{2,n}[v_i^1=v_s^1,v_j^2,v_t^2, v_4^1,v_5^1]\cong K_{2,3}$
and both $e$ and $f$ are in $G'_1$, $e\theta f$.

If $|A|=2$, then $e$ and $f$ are two non-adjacent edges. Let $i=t=1$ and $j=s=2$ for convenience.
Then $G_1=G_{2,n}[v_i^1,v_j^2,v_4^1, v_5^1,v_3^2]\cong K_{2,3}$ and $G_2=G_{2,n}[v_s^1,v_t^2,v_4^1, v_5^1,v_3^2]\cong K_{2,3}$.
Since $e\in E(G_1)$, $f\in E(G_2)$ and $v_5^1v_3^2\in E(G_1)\cap E(G_2)$, $e\theta f$.

Now we prove $md(P_3\ast K_3)=1$. The graphs $P_3,K_3$ and $P_3\ast K_3$ are shown as on the left-hand-side of Figure \ref{11},
and we write the vertex $(y_i,x_j)$ of $P_3\ast K_3$ as $v_i^j$. The planar embedding of $G=P_3\ast K_3$ is shown as on the right-hand-site of Figure \ref{11}.
We will complete the proof by checking all the possible edge-colorings of $P_3\ast K_3$.
\begin{figure}[h]
    \centering
    \includegraphics[width=270pt]{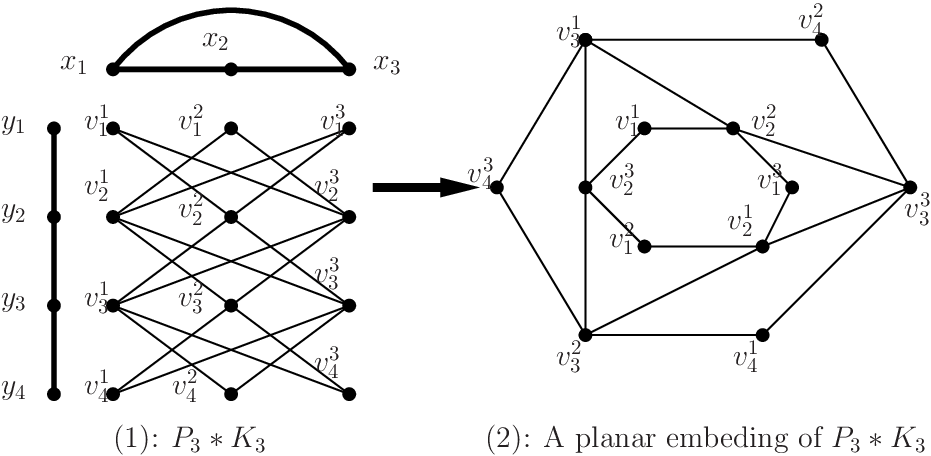}\\
    \caption{The graph $P_3\ast K_3$.} \label{11}
\end{figure}

The central cycle $C=G[v_1^1,v_2^2,v_1^3,v_2^1,v_1^2,v_2^3]$ of $G$ is crucial for our discussion. Since the opposite edges of $C_4$ are colored the same under its any $MD$-coloring,
$\Gamma(G)=\Gamma(C)$ for any $MD$-coloring of $G$. If $md(G)\geq 2$, by Lemma \ref{leqk}, there is an $MD$-coloring $\Gamma'$ of $G$ such that $|\Gamma'|=2$.
All possible edge-colorings of $C$ under $\Gamma'$ are shown as in Figure \ref{22} $A,B,C$ and $D$, and the colors of the other edges are also labeled.
If $\Gamma'$ is an edge-coloring shown as in Figure \ref{22} $A$, then $\Gamma'$ is not an $MD$-coloring restricted on the cycle $C_1=G[v_4^3,v_3^2,v_2^1,v_1^3,v_2^2,v_3^1]$;
if $\Gamma'$ is an edge-coloring shown as in Figure \ref{22} $B$, $C$ or $D$, then $\Gamma'$ is not an $MD$-coloring restricted on the cycle $C_2=G[v_2^3,v_3^2,v_4^1,v_3^3,v_2^2,v_3^1]$.
All the four cases contradict that $\Gamma'$ is an $MD$-coloring of $G$, and so $md(G)=1$.
\end{proof}

\begin{figure}[h]
    \centering
    \includegraphics[width=270pt]{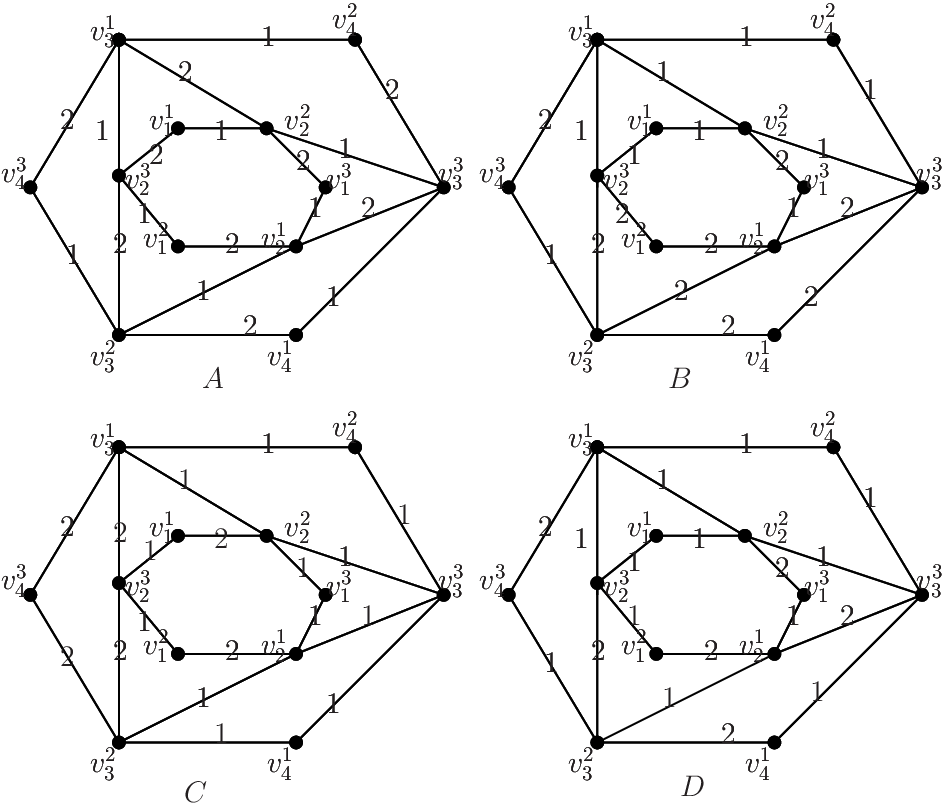}\\
    \caption{All possible $2$-edge-coloring of $P_3\ast K_3$.} \label{22}
\end{figure}

\begin{lemma}\label{GHK3}
Let $G$ and $H$ be two connected graphs and let $G'$ be a connected subgraph of $G$. If at least one of $G'$ and $H$ is non-bipartite graph and $\delta(H)\geq 2$, then $md(G\ast H)\leq md(G'\ast H)$.
\end{lemma}
\begin{proof}
We proceed the proof by induction on $|G|-|G'|$. If $|G|-|G'|=0$, then $G'$ is a spanning subgraph of $G$. This implies that $G'\ast H$ is a spanning subgraph of $G\ast H$.
Since at least one of $G'$ and $H$ is not bipartite, by Proposition \ref{Te}, both of $G\ast H$ and $G'\ast H$ are connected graphs. Then by Lemma \ref{sub}, $md(G\ast H)\leq md(G'\ast H)$,
and the result thus holds.

Now we suppose $|G|-|G'|\geq 1$. Since $G'$ is a connected subgraph of $G$, there is a spanning tree of $G$ such that one of its leaves, say $u$, is not in $V(G')$. Let $G^*=G-u$. Then $G^*$ is a connected subgraph of $G$ containing $G'$ as its subgraph. Furthermore, both of $G\ast H$ and $G^*\ast H$ are connected by Proposition \ref{Te}. Since $|G^*|-|G'|<|G|-|G'|$, by induction, $md(G^*\ast H)\leq md(G'\ast H)$.

Let $V(H)=\{w_1,w_2,\cdots,w_n\}$ and let $S=\{(u,w_i):i\in [n]\}$. Then $S$ is an independent set of $G\ast H$. Furthermore, $G\ast H-S=G^*\ast H$. For an element $(u,w)$ of $S$, since $\delta(H)\geq 2$,
there are two neighbors of $w$ in $H$, say $w_1$ and $w_2$. Let $v$ be a neighbor of $u$ in $G$. Then $((u,w),(v,w_1))$ and $((u,w),(v,w_2))$ are edges of $G\ast H$ incident with $(u,w)$.
Therefore, each vertex of $S$ has a degree at least two in $G\ast H$.
Let $\gamma=((u,w_1),\cdots,(u,w_n))$ be a vertex sequence of $G\ast H$. Then $\gamma$ is a soft-layer. By Lemma \ref{sequ}, $md(G\ast H)\leq md(G^*\ast H)$.
Since $md(G^*\ast H)\leq md(G'\ast H)$, $md(G\ast H)\leq md(G'\ast H)$.
\end{proof}

\begin{theorem}\label{GHK4}
Let $G'$ and $H'$ be connected subgraphs of the connected graphs $G$ and $H$, respectively, and all the four graphs do not have pendent edges.
If at least one of $G'$ and $H'$ is non-bipartite, then $md(G\ast H)\leq md(G'\ast H')$.
\end{theorem}
\begin{proof}
Since at least one of $G'$ and $H$ is non-bipartite and $\delta(H)\geq2$, by Lemma \ref{GHK3}, $md(G\ast H)\leq md(G'\ast H)$.
Analogously, since at least one of $G'$ and $H'$ is non-bipartite and $\delta(G')\geq2$, $md(H\ast G')=md(H'\ast G')=md(G'\ast H')$.
Therefore, $md(G\ast H)\leq md(G'\ast H')$.
\end{proof}

The {\em odd girth} of a non-bipartite graph $G$ is the length of a minimum odd cycle of $G$, and we denote it by $g_o(G)$.
If $G$ is a bipartite graph, we define $g_o(G)=+\infty$, this is because a bipartite graph has no odd cycle.

\begin{corollary}
Let $G$ and $H$ be two connected non-trivial graphs both without pendent edges and at least one of them is non-bipartite. Then $md(G\ast H)\leq \min\{g_o(G),g_o(H)\}$.
\end{corollary}
\begin{proof}
Without loss of generality, suppose $G$ contains an odd cycle $O$ such that $|O|=\min\{g_o(G),g_o(H)\}$. Since $H$ has no pendent edge, $H$ has a cycle $O'$. By Lemma \ref{GHK4},
$md(G\ast H)\leq md(O\ast O')$. By Lemma \ref{GHK3}, $md(O\ast O')\leq md(O\ast K_2)$. Since $O\ast K_2$ is a $(2|O|)$-cycle, $md(O\ast K_2)=|O|=\min\{g_o(G),g_o(H)\}$.
Therefore, $md(G\ast H)\leq md(O\ast K_2)= \min\{g_o(G),g_o(H)\}$.
\end{proof}
\begin{corollary}
Let $G$ and $H$ be two connected graphs. Then
\begin{enumerate}
\item if $G$ is neither a tree nor a unicycle graph with the cycle $K_3$, and $H$ contains a triangle but does not have pendent edges, then $md(G\ast H)=1$;
\item if $|G|\geq 2$ and $H=K_n$ where $n\geq 5$, then $md(G\ast H)=1$.
\end{enumerate}
\end{corollary}
\begin{proof}
We prove the first result.
Let $G'$ be a graph obtained from $G$ by deleting pendent edges one by one. Since $G$ is neither a tree nor a unicycle graph with the cycle $K_3$,
$G'$ has no pendent edges and is not a $K_3$. Therefore, $G'$ contains a $3$-path, say $P$. By Theorem \ref{GHK4}, $md(G\ast H)\leq md(G'\ast K_3)$.
By Lemma \ref{GHK3} and \ref{KP}, $md(G'\ast K_3)\leq md(P\ast K_3)=1$. So, $md(G\ast H)=1$.

Since $md(G\ast K_n)\leq md(K_2\ast K_n)$ and $md(K_2\ast K_n)=1$ for $n\geq 5$ by Lemma \ref{GHK3} and \ref{KP}, respectively, the second result can be derived directly.
\end{proof}

\end{document}